\documentclass[xypic,amssymb,rawfonts,amsthm,amsmath,amsfonts,syntonly,11pt]{amsart}

\usepackage{amssymb}
\usepackage{hyperref}

\usepackage[all]{xy}
\CompileMatrices

\newcounter{thecounter}
\numberwithin{thecounter}{section}
\newtheorem{lemma}[thecounter]{Lemma}
\newtheorem{prop}[thecounter]{Proposition}
\newtheorem{thm}[thecounter]{Theorem}
\newtheorem{cor}[thecounter]{Corollary}

\theoremstyle{definition}

\newtheorem{rem}[thecounter]{Remark}

\numberwithin{equation}{section}

\newcommand{\Hom}{\operatorname{Hom}}
\newcommand{\im}{\operatorname{im}}
\newcommand{\Tor}{\operatorname{Tor}}
\newcommand{\fd}{\operatorname{fd}}

\newcommand{\cC}{{\mathcal{C}}}
\newcommand{\pcom}{\hat{{}_p}}

\newcommand{\DI}{\operatorname{DI}}
\newcommand{\D}{{\mathbf{D}}}
\newcommand{\dT}{\breve T}

\newcommand{\GL}{\operatorname{GL}}
\newcommand{\SU}{\operatorname{SU}}
\newcommand{\Sp}{\operatorname{Sp}}
\newcommand{\Spin}{\operatorname{Spin}}

\newcommand{\Q}{{\mathbb{Q}}}
\newcommand{\F}{{\mathbb{F}}}
\newcommand{\Z}{{\mathbb{Z}}}
\newcommand{\C}{{\mathbb{C}}}

\newcommand{\CP}{{\mathbb{C}P}}
\newcommand{\cP}{{\mathcal{P}}}
\newcommand{\co}{\colon\thinspace}

\begin{document}

\title{The Steenrod problem of realizing polynomial cohomology rings}

\author[K. Andersen and J. Grodal]{Kasper K. S. Andersen and Jesper Grodal}

\thanks{The second named author was partially supported by NSF grant
  DMS-0354633, the Alfred P. Sloan Foundation, and the Danish Natural
  Science Research Council}


\address{Department of Mathematical Sciences, University of Aarhus,
  Ny Munkegade, Bygning 1530, DK-8000 Aarhus C, Denmark}
\email{kksa@imf.au.dk}

\address{
Department of Mathematical Sciences, University of Copenhagen,
Universitetsparken 5, DK-2100
Copenhagen, Denmark}
\email{jg@math.ku.dk}

\begin{abstract}
In this paper we completely classify which graded polynomial
$R$--algebras in finitely many even degree variables can occur as the
singular cohomology of a space with coefficients in $R$, a $1960$
question of N.~E.~Steenrod, for a commutative ring $R$ satisfying mild
conditions. In the fundamental case $R = \Z$, our result states
that the only polynomial cohomology rings over $\Z$ which can
occur, are tensor products of copies of $H^*(\CP^\infty;\Z) \cong \Z[x_2]$,
$H^*(B\SU(n);\Z) \cong \Z[x_4, x_6, \ldots, x_{2n}]$,
and $H^*(B\Sp(n);\Z) \cong \Z[x_4, x_8, \ldots, x_{4n}]$, confirming
an old conjecture.
Our classification extends Notbohm's solution for $R = \F_p$, $p$
odd. Odd degree generators, excluded above, only occur if $R$ is an
$\F_2$--algebra and in that case the recent classification of
$2$--compact groups by the authors can be used instead of the present
paper. Our proofs are short and rely on the general theory
of $p$--compact groups, but not on classification
results for these.
\end{abstract}

\maketitle


\section{Introduction}

In $1960$, N.~E.~Steenrod \cite{steenrod61} asked which graded
polynomial rings occur as the cohomology ring
of a space. In this paper we answer this question.

\begin{thm} \label{mainZ}
If $H^*(X;\Z)$ is a finitely generated polynomial algebra over $\Z$,
for some space $X$, then $H^*(X;\Z)$ is isomorphic, as a graded
algebra, to a tensor product of copies of $H^*(\CP^\infty;\Z) \cong \Z[x_2]$,
$H^*(B\SU(n);\Z) \cong
\Z[x_4, x_6, \ldots, x_{2n}]$, and $H^*(B\Sp(n);\Z) \cong \Z[x_4, x_8,
\ldots, x_{4n}]$.
\end{thm}

This has been a standard conjecture since the $1970$s. In fact we
solve the question over any ground ring
$R$, satisfying mild assumptions. As usual, $\CP^\infty$ denotes
infinite complex projective space, and
$B\SU(n)$ and $B\Sp(n)$ are the classifying spaces of the special
unitary group and the symplectic group respectively.
 We first describe the background of the
problem and previously known results.

Steenrod proved in his $1960$ paper \cite{steenrod61} that if $H^*(X;\Z)
\cong \Z[x]$ then $\left| x\right| =2$ or $4$, as
a consequence of his newly introduced Steenrod operations, 
settling the one-variable case over $\Z$. In
contrast, when taking coefficients in a field $R$ of characteristic zero every
graded polynomial ring on even degree generators can occur, since
$H^*(K(\Z,2n);R) \cong R[x_{2n}]$, as proved by Serre in his thesis
\cite[Ch.~VI, \S3, Prop.~4]{serre51}. Note that by anti-commutativity, all
generators have to be in even degrees, whenever $2\neq 0$ in $R$.

Restrictions on realizable polynomial rings over
$\F_p$ were studied in the $1960$s and $1970$s, mainly
by ingenious use of Steenrod operations
(see e.g., \cite{thomas63,thomas65,ST69,steenrod71}). Key progress was
made with the discovery of a connection with finite $p$--adic
reflection groups by Sullivan \cite{sullivan05},
Clark--Ewing \cite{CE74}, and others in the early $1970$s. Further significant progress was
made with the connection to the category of unstable modules
over the Steenrod algebra starting with Wilkerson \cite{wilkerson77}
and Adams--Wilkerson \cite{AW80} in the late $1970$s and early $1980$s.
Aguad{\'e} in his $1981$ paper \cite{aguade81} used Adams--Wilkerson's work
\cite{AW80} to obtain a partial version of Theorem~\ref{mainZ},
under the strong additional hypothesis that
all generators are in different degrees. (Theorem~\ref{mainZ} is referred to in \cite{aguade81} as ``the standing conjecture''; see also
the survey paper \cite{aguade82} for a historical account and
references as of $1981$.)
Following additional partial results, Notbohm \cite{notbohm99}
in $1999$ used
the more powerful techniques of $p$--compact groups, as developed by
Dwyer--Wilkerson \cite{DW94,DMW92} and others since the early $1990$s,
to obtain a full answer in the case where the ground ring is $\F_p$ for
$p$ odd. Note however that $2$--primary information is needed for
Theorem~\ref{mainZ} as e.g., $H^*(B\Spin(2n);\Z[\frac{1}{2}]) \cong
\Z[\frac12][x_4,x_8,\ldots,x_{4(n-1)},x_{2n}]$.

In the present paper we obtain
a solution for any commutative Noetherian ring $R$ of
finite Krull dimension, under the
assumption that all generators are in even degrees. The even
degrees assumption is automatic unless $R$ is an
$\F_2$--algebra; in that case a solution however
follows as a consequence of our recent
classification of $2$--compact groups; see
\cite[Thm.~1.4]{AG08classification}. We use standard results about
$p$--compact groups in the present paper, but remark that classification
results for these, as well as Notbohm's results in the
case $R=\F_p$, $p$ odd, are not used. We also
note that Theorem~\ref{mainZ} can alternatively be deduced from
our aforementioned result
\cite[Thm.~1.4]{AG08classification}, but the argument we present
here is significantly more direct, as explained below.
By the {\em type} of a graded polynomial ring we mean the multiset (i.e.,
unordered tuple) of degrees of its generators. The general result 
(implying Theorem~\ref{mainZ}) which we prove here is as follows.

\begin{thm} \label{mainthm}
Let $R$ be a commutative Noetherian ring of finite Krull dimension, $\cP$
the set of prime numbers $p$ which are not units in $R$, and $A$ a
graded polynomial $R$--algebra in finitely many variables, all in positive
even degrees. Then, there exists a space $Y$ such that
$A \cong H^*(Y;R)$ as graded $R$--algebras if and only if for each prime $p
\in \cP$ the type of $A$ is a union of multisets of degrees in
Tables~\ref{lietable}--\ref{extable2} which occur at the prime $p$.
\begin{table}[htb]
\addtolength\arraycolsep{13pt}
$$
\begin{array}{c|c|c}
\text{Lie group} & \text{Degrees} & \text{Occur for}\\
\hline
S^1 & 2 & \text{all $p$}\\
\SU(n) & 4, 6, \ldots, 2n & \text{all $p$}\\
\Sp(n) & 4, 8, \ldots, 4n & \text{all $p$}\\
\Spin(2n) & 4, 8, \ldots, 4(n-1),\, 2n & p\geq 3\\
G_2 & 4, 12 & p\geq 3\\
F_4 & 4, 12, 16, 24 & p\geq 5\\
E_6 & 4, 10, 12, 16, 18, 24 & p \geq 5\\
E_7 & 4, 12, 16, 20, 24, 28, 36 & p \geq 5\\
E_8 & 4, 16, 24, 28, 36, 40, 48, 60 & p \geq 7\\
\end{array}
$$
\caption{Lie group cases}
\label{lietable}
\end{table}
\vspace{-28pt}
\begin{table}[htb]
$$
\begin{array}{c|c|c|c}
W & \text{Degrees} & \text{Conditions} &\text{Occur for}\\
\hline
\begin{gathered}[t]
G(m,r,n) \\ 
D_{2m} \\
C_m 
\end{gathered}
&
\begin{gathered}[t]
2m, 4m, \ldots, 2(n-1) m,\, 2mn/r\\
4, 2m \\
2m
\end{gathered}
&
\begin{gathered}[t]
n\geq 2,\, m\geq 3,\, r\!\mid\!m\\
m\geq 5,\, m\neq 6\\
m\geq 3
\end{gathered}
&
\begin{aligned}[t]
p & \equiv 1 \pmod{m}\\
p & \equiv \pm 1 \pmod{m}\\
p & \equiv 1 \pmod{m}
\end{aligned}
\\
\end{array}
$$
\caption{Exotic cases, first part (family 2a, 2b, and 3)}
\label{extable1}
\end{table}
\nopagebreak
\vspace{-28pt}
\nopagebreak
\begin{table}[h]
\begin{gather*}
\begin{array}{cc}
\begin{array}{c|c|c}
W & \mbox{Degrees} & \mbox{Occur for} \\
\hline
\begin{gathered}[t]
\mathstrut G_8                                         \\
\mathstrut G_9                                         \\
\mathstrut G_{12}                                      \\
\mathstrut G_{14}                                      \\
\mathstrut G_{16}                                      \\
\mathstrut G_{17}                                      \\
\mathstrut G_{20}                                      \\
\mathstrut G_{21}                                      \\
\mathstrut G_{22}                                      
\end{gathered}
            & 
\begin{gathered}[t]
\mathstrut 16, 24                                      \\                
\mathstrut 16, 48                                      \\             
\mathstrut 12, 16                                      \\           
\mathstrut 12, 48                                      \\         
\mathstrut 40, 60                                      \\       
\mathstrut 40, 120                                     \\     
\mathstrut 24, 60                                      \\   
\mathstrut 24, 120                                     \\ 
\mathstrut 24, 40                                      
\end{gathered}
&
\begin{alignedat}[t]{2}
\mathstrut p &\equiv 1 &&\pmod{4}                      \\
\mathstrut p &\equiv 1 &&\pmod{8}                      \\
\mathstrut p &\equiv 1, 3 &&\pmod{8}                   \\
\mathstrut p &\equiv 1, 19 &&\pmod{24}                    \\
\mathstrut p &\equiv 1 &&\pmod{5}                      \\
\mathstrut p &\equiv 1 &&\pmod{20}                     \\
\mathstrut p &\equiv 1, 4 &&\pmod{15}                     \\
\mathstrut p &\equiv 1, 49 &&\pmod{60}                    \\
\mathstrut p &\equiv 1, 9 &&\pmod{20}                  
\end{alignedat}
\end{array}
&
\begin{array}{c|c|c}
W & \mbox{Degrees} & \mbox{Occur for} \\
\hline
\begin{gathered}[t]
\mathstrut G_{23}                                      \\
\mathstrut G_{24}                                      \\
\mathstrut G_{29}                                      \\
\mathstrut G_{30}                                      \\
\mathstrut G_{31}                                      \\
\mathstrut G_{32}                                      \\
\mathstrut G_{33}                                      \\
\mathstrut G_{34}
\end{gathered}
            & 
\begin{gathered}[t]
\mathstrut 4, 12, 20                                   \\
\mathstrut 8, 12, 28                                   \\
\mathstrut 8, 16, 24, 40                               \\
\mathstrut 4, 24, 40, 60                               \\
\mathstrut 16, 24, 40, 48                              \\
\mathstrut 24, 36, 48, 60                              \\
\mathstrut 8, 12, 20, 24, 36                           \\
\mathstrut 12, 24, 36, 48, 60, 84
\end{gathered}
&
\begin{alignedat}[t]{2}
\mathstrut p &\equiv 1, 4 && \pmod{5}                      \\
\mathstrut p &\equiv 1, 2, 4 &&\pmod{7},\, p\neq 2     \\
\mathstrut p &\equiv 1 &&\pmod{4}                      \\
\mathstrut p &\equiv 1, 4 &&\pmod{5}                   \\
\mathstrut p &\equiv 1 &&\pmod{4}                      \\
\mathstrut p &\equiv 1 &&\pmod{3}                      \\
\mathstrut p &\equiv 1 &&\pmod{3}                      \\
\mathstrut p &\equiv 1 &&\pmod{3}                      \\
\end{alignedat}
\end{array}
\end{array}
\end{gather*}
\caption{Exotic cases, continued}
\label{extable2}
\end{table}
\end{thm}

\pagebreak
The space $Y$ is rarely unique; see Remark~\ref{non-uniqueness}.
Moreover, the assumption that $R$ is Noetherian with finite
Krull dimension can be replaced by the
assumption that $Y$ has finite type; cf.~Proposition~\ref{mod-p-red}.

We briefly explain the origin of Tables~\ref{lietable}--\ref{extable2}:
Table~\ref{lietable} lists the simple simply connected compact Lie
groups and $S^1$, and the primes for which the homology of the group
is $p$--torsion free, together with the degrees of the polynomial
generators of the $\F_p$--cohomology of their classifying spaces, but
with $\Spin(2n+1)$, $p$ odd, left out for simplicity, since its
$\F_p$--cohomology agrees with that of $\Sp(n)$. This information goes
back to Borel \cite{borel61} (see Proposition~\ref{lieprop}).

To explain Tables~\ref{extable1} and
\ref{extable2} we recall some facts about reflection groups. The
Shephard--Todd--Chevalley theorem \cite[Thm.~7.2.1]{benson93} says that,
for a field $K$ of characteristic zero and a $K$--vector space $V$, a
finite group $W \leq \GL(V)$ is a reflection group if and only if the
ring $K[V]^W$ of $W$--invariant polynomial functions on $V$ is a
graded polynomial ring, where we grade $K[V]$ by giving the elements of
$V^*$ degree $2$. The {\em degrees} of $W \leq \GL(V)$ is defined to
be the type of this invariant ring. (We warn
the reader that degrees are sometimes defined as half of what is the
convention in this paper.)
Shephard--Todd \cite{ST54} classified the finite irreducible complex reflection
groups as falling into $3$ infinite families (labeled $1$--$3$) and
$34$ sporadic cases (labeled $G_i$, $4\leq i\leq 37$). Clark--Ewing
\cite{CE74} used this to give
a classification of finite
$\Q_p$--reflection groups (with additional clarification by Dwyer--Miller--Wilkerson
\cite[Prop.~5.5, Pf.~of~Thm.~1.5]{DMW92}). This classification says
that for a given prime $p$, a finite $\C$--reflection group gives
rise to a unique $\Q_p$--reflection group if and only if
its character field embeds in $\Q_p$, and all finite
$\Q_p$--reflection groups arise this way. Tables~\ref{extable1} and
\ref{extable2} essentially list the finite irreducible $\Q_p$--reflection
groups which are {\em exotic}, i.e., those whose character field is not
$\Q$, {\em except} the $\Q_2$--reflection group $G_{24}$; but for
simplicity we have also removed the groups $G_i$ for $i =
4$, $5$, $6$, $7$, $10$, $11$, $13$, $15$, $18$, $19$, $25$, $26$, and
$27$, whose degrees is readily obtainable as a multiset union of
the degrees of the remaining groups, at any prime for
which they exist as $\Q_p$--reflection groups.

To illustrate Theorem~\ref{mainthm}, if $\cP = \emptyset$ (i.e., if
$R$ is a $\Q$--algebra) there are no restrictions,
recovering Serre's result mentioned earlier. If $2 \in \cP$, the
possible types of $A$ are unions of the multisets $\{2\}$,
$\{4, 6, \ldots, 2n\}$, and $\{4, 8, \ldots, 4n\}$, in particular
recovering Theorem~\ref{mainZ}. If $\cP =\{3\}$ the type of $A$ is
a union of $\{2\}$, $\{4, 6, \ldots, 2n\}$,
$\{4, 8, \ldots, 4n\}$, $\{4, 8, \ldots, 4(n-1), 2n\}$, $\{4, 12\}$, and
$\{12, 16\}$, and similar lists can easily be compiled for $\cP =
\{p\}$ for any individual prime $p$. For arbitrary $\cP$, it is a
simple combinatorial problem to check whether any
given multiset of degrees is realizable, as is described in the following corollary.

\begin{cor} \label{arithprog}
For any multiset of degrees $\{2d_1, \ldots, 2d_r\}$, there exist
integers $a_1, \ldots, a_m$ and $N$ such that the following conditions are equivalent
for any commutative Noetherian ring $R$ of finite Krull dimension:
\begin{enumerate}
\item
The graded polynomial $R$--algebra $R[x_1,\ldots,x_r]$, with
$\left|x_i\right| = 2d_i$, is isomorphic to $H^*(Y;R)$ for some space $Y$.
\item
Every prime number $p$ which is not a unit in $R$ satisfies $p \equiv
a_i \pmod{N}$ for some~$i$.
\end{enumerate}
\end{cor}

We give an algorithm for finding $N$ and $a_1,\ldots,a_m$ in the
course of the proof. Finding explicit generators for the
commutative monoid of realizable multisets of degrees for a given collection
of primes $\cP$ is a harder combinatorial problem in general.

\medskip
The proof of the main theorem spans three short sections. The first two sections
reduce the problem to $p$--compact groups: In
Section~\ref{RtoFp-section} we show that if $H^*(Y;R)$ is a
polynomial algebra then the same holds for $H^*(Y;\F_p)$ for all prime numbers
$p$ which are not units in $R$, and hence that the $\F_p$--completion
$Y\pcom$ is the classifying space of a
$p$--compact group---the proof has a small ring theoretic twist due to
finite-type problems coming from the fact that we are working with
cohomology. Conversely Section~\ref{FptoR-section} constructs a space
$Y$ with polynomial $R$--cohomology from a collection of
$p$--compact groups with compatible cohomology rings, for $p$ any
non-unit in $R$.
Finally, in Section~\ref{pcg-section}, we determine the graded
polynomial $\F_p$--algebras on even degree generators which occur as
the $\F_p$--cohomology of the classifying space of a $p$--compact
group.
The key Proposition~\ref{polyprop} reduces this question
to separate questions for classifying spaces of compact Lie groups and
classifying spaces of exotic $p$--compact groups (without using
classification results on $p$--compact groups). These are then easily
solved, the first one essentially by Borel, and the second just using
earlier existence results for exotic $p$--compact groups. Combining the
earlier steps now easily establishes the main result, Theorem~\ref{mainthm},
and its corollaries. We remark that Notbohm's approach to the case $R
= \F_p$, $p$ odd, in \cite{notbohm99} differs from the one given here
in Section~\ref{pcg-section}.
Notbohm instead uses an argument relying on various earlier fairly
elaborate case-by-case calculations of invariant rings, which in
particular does not extend to the case of $\F_2$.

\medskip
\noindent
{\em Notation and recollections:} 
We stress that by a graded polynomial algebra we mean a graded algebra
generated by a set of homogeneous elements, which are algebraically
independent (i.e., we do not allow odd degree exterior
generators). For the theorems of the paper a
space can be taken to mean any topological space, though for the
purposes of the proofs we may without loss of generality restrict
ourselves to CW--complexes (or even simplicial sets), as we will often
tacitly do. We constantly use the theory of $p$--compact groups, and
here recall some pertinent facts, referring to e.g.,
\cite{DW94,dwyer98,AGMV08,AG08classification} for more information. A
$p$--compact group consists of a
triple $(X,BX,e\co X \xrightarrow{\simeq} \Omega BX)$ such that $BX$ is an
$\F_p$--complete space and $X$ has finite
$\F_p$--cohomology. Proposition~\ref{polylemma} recalls that spaces
with polynomial $\F_p$--cohomology ring
always come from $p$--compact
groups. A standard result, central to this paper, states that
for a connected $p$--compact group $X$,
$H^*(BX;\F_p)$ is a polynomial algebra concentrated in
even degrees if and only if $H^*(X;\Z_p)$ is torsion free if and only
if $H^*(BX;\Z_p) \xrightarrow{\cong} H^*(BT;\Z_p)^{W_X}$,
where $T$ and $W_X$ are respectively the maximal torus and the Weyl
group of $X$. (Parts of this result are due to Borel and parts to
Dwyer--Miller--Wilkerson; see \cite[Thm.~12.1]{AGMV08}).)
In this case the type of $H^*(BX;\F_p)$ equal the
degrees of $W_X$, which act as a reflection group on $\pi_2(BT)
\otimes \Q$ with invariant ring $\Q_p[\pi_2(BT) \otimes \Q]^{W_{X}}
\cong H^*(BX;\Z_p) \otimes \Q$ \cite[Thm.~9.7]{DW94}---this provides
the link between polynomial cohomology rings concentrated
in even degrees, torsion
free $p$--compact groups, and finite $\Z_p$-- and $\Q_p$--reflection
groups. A couple of times in the proofs we use the word root
datum for which we refer to \cite[Section~8]{AG08classification},
although we do not use anything which is not already in
\cite{dw:center} in a slightly different language.

\medskip
\noindent{\em Acknowledgments:}
We thank Bill Dwyer for helpful conversations and Hans-Bj{\o}rn
Foxby for reminding us of what he taught us as undergraduates
about finitistic dimensions of commutative Noetherian rings.
We also thank Allen Hatcher for his interest and the referee for
helpful comments.


\section{From $R$ to $\F_p$} \label{RtoFp-section}

In this section we show how a polynomial cohomology ring over $R$
produces polynomial cohomology rings over $\F_p$ for the prime numbers
$p$ which are not units in $R$.

\begin{prop} \label{RtoFp}
Suppose that $R$ is a commutative Noetherian ring of finite Krull
dimension, and $Y$ is a space such that $H^*(Y;R)$ is a polynomial
$R$--algebra in finitely many variables, each in positive degree. If a
prime $p$ is not a unit in $R$ then $H^*(Y;\F_p)$ is a polynomial
$\F_p$--algebra with the same type and $Y\pcom$ is
the classifying space of a $p$--compact group.
\end{prop}

The proof will occupy the rest of this section. The proof is
straightforward if e.g.\ $R\subseteq \Q$ (using the
exact sequence $0 \to R \xrightarrow{p} R \to \F_p \to 0$), but in
general it is a bit more subtle, with complications arising even for
$R = \Z/4$.
First we explain the well known fact that spaces with polynomial
$\F_p$--cohomology ring give rise to $p$--compact groups.

\begin{prop} \label{polylemma}
Let $Y$ be a space. If $H^*(Y;\F_p)$ is a finitely generated
polynomial algebra over $\F_p$, then $Y\pcom \simeq BX$ for a
$p$--compact group $X$. If furthermore $H^1(Y;\F_p)=0$, then $X$ is
connected.
\end{prop}

\begin{proof}
By considering the fiber of the canonical map $Y \to K(H_1(Y;\F_p),1)$
one easily reduces to the case where $H_1(Y;\F_p) = 0$
(cf.~\cite[Pf.~of~Thm.~1.4]{AG08classification}), and this is in fact
also the only case relevant for our main theorem. By
\cite[Prop.~VII.3.2]{bk} $Y$ is $\F_p$--good and $Y\pcom$ is
$\F_p$--complete and simply connected. In particular the
Eilenberg--Moore spectral sequence of the path-loop fibration of
$Y\pcom$ converges \cite{dwyer74} and shows that $H^*(\Omega
Y\pcom;\F_p)$ is finite dimensional over $\F_p$. Hence $Y\pcom$ is the
classifying space of a $p$--compact group, which is connected since
$Y\pcom$ is simply connected.
\end{proof}

Next we deal with the finiteness restrictions
usually associated with universal coefficient theorems in
cohomology. Recall that the {\em finitistic flat dimension}
of a commutative ring $R$ is defined to be the supremum over the flat
dimensions of all $R$--modules with finite flat dimension.

\begin{lemma} \label{finitisticlemma}
Let $R$ be a commutative ring, and $C_*$ a chain complex of
$R$--modules (with differential of degree $-1$) concentrated in non-positive degrees such that $C_n$ and
$H_n(C_*)$ are flat for all $n$. If the finitistic flat dimension of
$R$ is finite, then $H_n(C_*) \otimes_R M \xrightarrow{\cong} H_n(C_*
\otimes_R M)$ for any $R$--module $M$.
\end{lemma}

\begin{proof}
Let $\fd(-)$ denote the flat dimension of an $R$--module. By the
K{\"u}nneth formula \cite[Thm.~3.6.1]{weibel94} it suffices to show
that $\im(C_n \xrightarrow{d_n} C_{n-1})$ is flat for all $n$. To see
this, first note that by the short exact sequence
$$
0 \to \im(d_{n+1}) \to \ker(d_{n}) \to H_{n}(C_*) \to 0,
$$
we have $\fd\bigl(\im(d_{n+1})\bigr) = \fd\bigl(\ker(d_{n})\bigr)$, since
$H_{n}(C_*)$ is flat, cf. \cite[Ex.~4.1.2(3)]{weibel94}. Now, since
$C_{n}$ is flat, the short exact sequence
$$
0 \to \ker(d_{n}) \to C_{n} \to \im(d_{n}) \to 0
$$ 
shows that either $\im(d_n)$ and $\im(d_{n+1})$ are flat or
$\fd\bigl(\im(d_{n+1})\bigr) = \fd\bigl(\ker(d_{n})\bigr) =
\fd\bigl(\im(d_{n})\bigr)-1$.
Since $\im(d_n)=0$ is flat for $n$ positive, it follows that either $\im(d_n)$
is flat for all $n$ or there exists an $n_0$ such that $\im(d_n)$ is
flat for $n\geq n_0$ and $\fd\bigl(\im(d_n)\bigr) = n_0-n$ for $n\leq
n_0$. Since $R$ has finite finitistic flat dimension the second possibility
cannot occur, and we are done.
\end{proof}

\begin{rem}
A way to view the assumptions in the lemma is that they ensure
convergence of the K{\"u}nneth spectral sequence, derived from the
spectral sequence of a double complex (see \cite[\S5.6]{weibel94}). In
complete generality the K{\"u}nneth spectral sequence need not
converge, the standard counterexample being $R =\Z/4$, $C_* = \cdots
\xrightarrow{2} \Z/4 \xrightarrow{2} \Z/4 \xrightarrow{2} \cdots$ and
$M=\Z/2$.
\end{rem}

The preceding lemma, together with a result
in commutative ring theory, gives the
following result about the cohomology of spaces.

\begin{prop} \label{mod-p-red}
Let $R$ be a commutative ring and $Y$ a space such that $H^*(Y;R)$ is finite
free over $R$ in each degree.
If either $R$ is Noetherian with finite Krull dimension or $Y$ has
finite type then
$$
H^*(Y;R) \otimes_R R/p \xrightarrow{\cong} H^*(Y;R/p)
\xleftarrow{\cong} H^*(Y;\F_p) \otimes_{\F_p} R/p
$$
as $R/p$--algebras.
\end{prop}

\begin{proof} 
Assume first that $Y$ is arbitrary and that
that $R$ is Noetherian with finite Krull dimension.
Then the singular cochain complex
$C^*(Y;R)$ is a complex of flat modules, since over a commutative
Noetherian ring arbitrary products of flat modules are again flat
(i.e., commutative Noetherian rings are coherent; see
\cite[Thm.~2.1]{chase60} or \cite[Ex.~VI.4, p. 122]{CE56}). By a
result of Auslander--Buchsbaum \cite[Thm.~2.4]{AB58}, the finitistic
flat dimension of $R$ is bounded above by its Krull dimension
(in fact they differ by at most $1$ by a result of Bass
\cite[Cor.~5.3]{bass62}), and in particular the assumptions imply that
it is finite. Hence, since obviously $C^*(Y;R) \otimes_R R/p \cong
C^*(Y;R/p)$, Lemma~\ref{finitisticlemma} implies that $H^*(Y;R)
\otimes_R R/p \xrightarrow{\cong} H^*(Y;R/p)$. Now
$$
H^*(Y;R/p) = H(C^*(Y;R/p)) = H(\Hom_{\F_p}(C_*(Y;\F_p),R/p))
\xrightarrow{\cong} \Hom_{\F_p}(H_*(Y;\F_p),R/p)
$$
as $R/p$--modules, and in particular $H_*(Y;\F_p)$ is finite in each
degree. Hence $H^*(Y;\F_p) \otimes_{\F_p} R/p \xrightarrow{\cong}
\Hom_{\F_p}(H_*(Y;\F_p),R/p)$, which combined with the previous
isomorphisms gives the result under the ring theoretic assumption on $R$. 

Now assume $Y$ has finite type, and that $R$ is an arbitrary
commutative ring. Let $C^*(Y;-)$ denote the cellular cochain
complex, and note that, by the finite type assumption, $C^*(Y;\Z) \otimes_{\Z} R
\xrightarrow{\cong} C^*(Y;R)$. Now  $H^*(Y;R) \otimes_R R/p \xrightarrow{\cong} H^*(Y;R/p)$
by  a result of Dold \cite[Satz~5.2]{dold62}, applied to the $\Z$--chain complex $C_* =
C^{-*}(X;Z)$, $\Lambda =R$ and $M =R/p$. (Note that the proof of
Dold's theorem is more involved than the proof of the ordinary K{\"u}nneth
theorem.) The second isomorphism in the proposition now follows as above.
\end{proof}

\begin{lemma} \label{poly-check}
Let $A$ be a graded connected $\F_p$--algebra which is finite
dimensional in each degree, and $B$
an $\F_p$--algebra (viewed as a graded algebra concentrated in degree
$0$). If $A \otimes_{\F_p} B$ is a graded polynomial
$B$--algebra, then $A$ is a graded polynomial $\F_p$--algebra with the
same type.
\end{lemma}

\begin{proof}
By reducing $B$ modulo a maximal ideal, we can without restriction assume
that $B$ is a field. Since $B$ is a flat $\F_p$--module, $Q(A)
\otimes_{\F_p} B \xrightarrow{\cong} Q(A \otimes_{\F_p} B)$, where
$Q(\cdot)$ denotes the
module of indecomposable elements. Picking a basis for the
$\F_p$--vector space $Q(A)$ and lifting these to $A$ produces a set of
generators for $A$ as an $\F_p$--algebra, which maps to a set of generators
for $A \otimes_{\F_p}B$ as a $B$--algebra. By construction they form
a $B$--basis for $Q(A \otimes_{\F_p} B)$. Hence they have to be algebraically
independent, since $A \otimes_{\F_p} B$ is assumed to be a graded polynomial
$B$--algebra (to see this note that mapping polynomial generators of
$A \otimes_{\F_p} B$ to the constructed generators will produce an
epimorphism $A \otimes_{\F_p} B \to A \otimes_{\F_p} B$ which also has to be a
monomorphism since $A \otimes_{\F_p} B$ is finite dimensional over $B$ in each
degree). Since $B$ is a faithfully flat $\F_p$--module, the generators are
algebraically independent in $A$ as well, and hence $A$ is a graded
polynomial $\F_p$--algebra.
\end{proof}

\begin{proof}[Proof of Proposition~\ref{RtoFp}]
By Proposition~\ref{mod-p-red} and Lemma~\ref{poly-check},
$H^*(X;\F_p)$ is a polynomial algebra with the same type as
$H^*(X;R)$. It hence follows from Proposition~\ref{polylemma} that $X$
is $\F_p$--good and $X\pcom$ is the classifying space of a
$p$--compact group.
\end{proof}


\section{From $\F_p$ to $R$} \label{FptoR-section}

In this section we show how spaces with polynomial $\F_p$--cohomology
ring at different primes can be glued together, provided that they
have the same type.

\begin{prop} \label{glueing}
Let $\cP$ be a set of prime numbers, $J$ the set of prime numbers not
in $\cP$, and
$\{2d_1,\ldots,2d_r\}$ a fixed multiset. Suppose that for each $p \in
\cP$ there exists a space $B_p$ such that $H^*(B_p;\F_p)$ is a
polynomial algebra with type $\{2d_1,\ldots,2d_r\}$. Then there exists a
simply connected finite type space $Y$, such that
$H^*(Y;\Z[J^{-1}])$ is a polynomial algebra over $\Z[J^{-1}]$ with
generators in degrees $\{2d_1,\ldots,2d_r\}$.
More generally, if $R$ is a commutative ring with $R = R[J^{-1}]$,
then $H^*(Y;R)$ is also a polynomial $R$--algebra with generators in
the same degrees.
\end{prop}

Before the proof we need the following lemma, which is similar to a
lemma used by Baker--Richter \cite[Prop.~2.4]{BR06} in a different
context.

\begin{lemma} \label{poly-at-each-p}
Let $J$ be a set of prime numbers and set $R = \Z[J^{-1}]$. Suppose
that $A$ is a graded connected $R$--algebra, finite free over
$R$ in each degree. If $A \otimes_R \Q$ is a graded polynomial $\Q$--algebra and
$A \otimes_R \Z_p$ is a graded polynomial $\Z_p$--algebra for all primes $p \not
\in J$, then $A$ is a graded polynomial $R$--algebra as well (with the
same type as both $A \otimes_R \Z_p$, $p\not\in J$, and $A \otimes_R \Q$).
\end{lemma}

\begin{proof} 
The proof is similar to the proof of Lemma~\ref{poly-check}. Since $\Z_p$ is flat
over $R$ for $p\not\in J$, we have $Q(A) \otimes_R \Z_p
\xrightarrow{\cong} Q(A \otimes_R \Z_p)$, which is free over $\Z_p$ in
each degree by assumption. Hence $Q(A)$ does not have $p$--torsion when
$p\not\in J$, and since $Q(A)$ is finite over $R$ in each degree, we conclude by
the structure theorem for finite $\Z[J^{-1}]$--modules that $Q(A)$
is finite free over $R$ in each degree. Likewise, since $\Q$ is flat
over $R$, $Q(A) \otimes_R \Q \xrightarrow{\cong} Q(A \otimes_R
\Q)$. Pick homogeneous elements in $A$ which project to an $R$--basis
of $Q(A)$. Clearly these elements generate $A$ as an
$R$--algebra. They are also algebraically
independent, since they, via the isomorphism $Q(A) \otimes_R \Q
\xrightarrow{\cong} Q(A \otimes_R \Q)$, give rise to generators of the
graded polynomial ring $A \otimes_R \Q$. Hence $A$ is a graded polynomial
$R$--algebra, with the same type as $A \otimes_R \Q$ and
$A \otimes_R \Z_p$, $p\not\in J$.
\end{proof}

\begin{proof}[Proof of Proposition~\ref{glueing}]
The proof is reminiscent of \cite[Pf.~of~Lem.~1.2]{ABGP03}, to which we
also refer. By Proposition~\ref{polylemma}, $(B_p)\pcom$ is
$\F_p$--complete with the same $\F_p$--cohomology as $B_p$, so we may
assume that $B_p$ is $\F_p$--complete. Set $K = K(\Z[\cP^{-1}],2d_1) \times \cdots \times
K(\Z[\cP^{-1}],2d_r)$. We will define $Y$ as a homotopy pull-back
\begin{equation} \label{pullbackdiag}
\xymatrix{
Y \ar[r] \ar[d] & {\prod_{p\in \cP} B_p} \ar[d] \\
K \ar[r]^-f & (\prod_{p \in \cP} B_p)_\Q
}
\end{equation}
where the map $f$ has to be constructed.

Since $B_p$ is the classifying space of a connected $p$--compact group by
Proposition~\ref{polylemma}, $\pi_*(B_p)$ and $H^*(B_p;\Z_p)$ are
finite over $\Z_p$ in each degree. It follows that $H^*(B_p;\Z_p)$ is also
a polynomial algebra over $\Z_p$, with type $\{2d_1,\ldots,2d_r\}$.
By an essentially classical result (see \cite[Thm.~2.1]{ABGP03})
$$
\pi_n(B_p) \cong \pi_{n-1}\left((S^{2d_1-1}
\times \cdots \times S^{2d_r-1})\pcom\right)
$$
for
$p>\max\{d_1,\cdots,d_r\}$ and similarly for the $\Q$--localization, so
$$
\pi_n\Bigl(\bigl(\prod_{p \in \cP} B_p\bigr)_\Q\Bigr) \cong
\pi_{n-1}\left(S^{2d_1-1} \times
\cdots \times S^{2d_r-1}\right) \otimes
\left(\prod_{p\in\cP}\Z_p\right) \otimes \Q
$$
Since these homotopy groups are concentrated in even degrees,
$\left(\prod_{p \in \cP} B_p\right)_\Q$ is a product of
Eilenberg--Mac\,Lane spaces, as follows from standard obstruction
theory \cite[Ch.~IX]{whitehead78}.

Hence we can define $f\co K \to (\prod_{p \in \cP} B_p)_\Q$ to be the
canonical map induced by the unique ring map $\Z[\cP^{-1}] \to \Q \to
(\prod_{p\in\cP}\Z_p)\otimes\Q$.
By the Mayer--Vietoris sequence in homotopy applied to the pull-back
square \eqref{pullbackdiag} we see that 
$$
\pi_n(Y) = \left(\bigoplus_{i, 2d_i=n} \Z\right) \oplus
\left(\bigoplus_{p \in \cP} \Tor\left(\Q/\Z,\pi_n(B_p)\right)\right)
$$
In particular $\pi_n(Y)$ is finite over $\Z$, since for
a fixed $n$, $\pi_n(B_p)$ is $p$--torsion free, for $p$ large
\cite[Ch.~V, \S5, Prop.~4]{serre51}. Since $Y$ is simply connected
this implies that $Y$ is homotopy equivalent to a CW--complex with
finitely many cells in each dimension, i.e., $Y$ can be chosen to have
finite type (cf.\ e.g.\ \cite[Thm.~A]{wall65}).

For $p \in \cP$, the pull-back square above shows that the map $Y
\to B_p$ induces an isomorphism on mod $p$ homology so
$H^*(Y;\Z_p) \xleftarrow{\cong} H^*(B_p;\Z_p)$. Hence $H^*(Y;\Z_p)$
is a polynomial ring over $\Z_p$ with the same type as
$H^*(B_p;\F_p)$. Since
$H_*(Y;\Z[J^{-1}])$ is finite over $\Z[J^{-1}]$ in each
degree, the universal coefficient theorem
\cite[3.6.5]{weibel94} (applied to the $\Z[J^{-1}]$--module
$\Z_p$) shows that $H_*(Y;\Z[J^{-1}])$ is in fact finite free over
$\Z[J^{-1}]$ in each degree, and hence so is $H^*(Y;\Z[J^{-1}])$. In particular
$$
H^*(Y;\Z[J^{-1}]) \otimes_{\Z[J^{-1}]} \Z_p \cong H^*(Y;\Z_p) \cong H^*(B_p;\Z_p).
$$
Lemma~\ref{poly-at-each-p} now shows that $H^*(Y;\Z[J^{-1}])$ is a
polynomial $\Z[J^{-1}]$--algebra as wanted. Finally, let $R$ be a
commutative ring with $R=R[J^{-1}]$. Since $H_*(Y;\Z[J^{-1}])$
is finite free over $\Z[J^{-1}]$ in each degree, the universal
coefficient theorem again implies that $H^*(Y;R) \cong
H^*(Y;\Z[J^{-1}]) \otimes_{\Z[J^{-1}]} R$, proving the last
claim.
\end{proof}

\begin{rem} \label{non-uniqueness}
The space $Y$ in Proposition~\ref{glueing} is rarely unique, even up
to $\Z[J^{-1}]$--localization, due to a multitude of factors: First,
given a multiset of degrees and a prime $p$, there may be several
choices of $p$--compact groups with the same degrees;  this is a
finite, essentially combinatorial, problem (see
\cite[Thm.~1.5]{AG08classification} and \cite[Cor.~1.8]{notbohm99}).

Secondly, the classifying space $BX$ of a connected $p$--compact group
may lift to many $\Z_{(p)}$--homotopy types of spaces of finite type
over $\Z_{(p)}$, with examples at least going back to Bousfield and
Mislin \cite[Appendix]{NS90II}. This uses a variant of Wilkerson's
double coset formula \cite[Proof~of~Thm.~3.8]{wilkerson76} resulting
from the pull-back $K \to (BX)_\Q \leftarrow BX$, for $K$ a rational
space with the same rational degrees as $BX$. In fact a more careful
analysis along the same lines reveals that $BX$ has a unique lifting
as above if and only if $X$ is a $p$--compact torus or $X$ is a product
of rank one $p$--compact groups and the rational degrees
$\{2d_1,\ldots,2d_r\}$ has the property that none of the degrees can
be written as a non-negative integral linear combination of the
others. Moreover, if the lift is non-unique there are uncountably many lifts.

Thirdly, given a set of $\Z_{(p)}$--homotopy types $Y_p$ as in the
second step, one for each prime in $\cP$ (i.e., $p \notin J$), with
the same rational degrees, Zabrodsky mixing associated to the
pull-back $K \to (\prod_{p \in \cP} Y_p)_\Q \leftarrow \prod_{p \in
  \cP}Y_p$ in general produces infinitely many $\Z[J^{-1}]$--homotopy
types. Again a more careful analysis with the corresponding double
coset formula reveals that if $2 \leq |\cP| < \infty$, then there are
countably infinitely many lifts, unless $(Y_p)\pcom$ satisfies the
conditions of the second step for each $p \in \cP$, in which case
there is only one lift. If $\cP$ is infinite then there are
uncountably many lifts unless the rational degrees are
$\{2, 2, \ldots, 2\}$, i.e., the torus case. (See also e.g.,
 \cite{rector71,wilkerson76,NS90II,moller92,mcgibbon94} for related
 information.)
\end{rem}


\section{$p$--compact groups with polynomial $\F_p$--cohomology
  concentrated in even degrees} \label{pcg-section}

In this section we provide the necessary results on $p$--compact
groups, and use this, together with the results of the previous
sections, to prove Theorem~\ref{mainthm} and its corollaries.
The key step is the following result, which in fact also holds
without the even degree assumption
\cite[Thm.~1.4]{AG08classification}, though we do not know a proof in
that generality
which does not go via the classification of $2$--compact groups
\cite[Thm.~1.2]{AG08classification}.

\begin{prop} \label{polyprop}
If $X$ is a $p$--compact group such that $H^*(BX;\F_p)$ is a
finitely generated polynomial algebra over $\F_p$ concentrated in even
degrees, then we can write
$$
H^*(BX;\F_p) \cong H^*(BG;\F_p) \otimes H^*(BY;\F_p)
$$
as $\F_p$--algebras (and in fact as algebras over the Steenrod
algebra), for $G$ a compact connected Lie group and $Y$ a product of
exotic $p$--compact groups.
\end{prop}

We first need a lemma.

\begin{lemma} \label{conn}
Let $X$ be a $p$--compact group such that $H^*(BX;\F_p)$ is a
finitely generated polynomial algebra over $\F_p$ concentrated in even
degrees. Then $\cC_X(\nu)$ is connected for any elementary abelian
$p$--subgroup $\nu\co BE \to BX$ of $X$.
\end{lemma}

\begin{proof}
Let $\nu\co BE \to BX$ be an elementary abelian $p$--subgroup of $X$. By
\cite[Pf.~of~Thm.~8.1]{dw:center} we have $H^*(B\cC_X(\nu);\F_p) \cong
T_{E,\nu^*}(H^*(BX;\F_p))$, where $T_{E,\nu^*}$ denotes the component
of Lannes' $T$--functor corresponding to $\nu^*$
\cite[2.5.2]{lannes92}. Since $T_{E,\nu^*}$ preserves objects
concentrated in even degrees \cite[Prop.~2.1.3]{lannes92} as well as
finitely generated graded polynomial algebras \cite[Thm.~1.2]{DW98}, we
conclude that $H^*(B\cC_X(\nu);\F_p)$ is also a finitely generated
polynomial algebra over $\F_p$ concentrated in even degrees and in
particular $\cC_X(\nu)$ is connected by Proposition~\ref{polylemma}.
\end{proof}

\begin{proof}[{Proof of Proposition~\ref{polyprop}}]
By Proposition~\ref{polylemma}, $X$ is connected. The
classification of $\Z_p$--root data \cite[Thm.~8.1]{AG08classification}
says that the $\Z_p$--root datum $\D_X$ of
$X$ can be written as a direct product $\D_X \cong (\D_G \otimes \Z_p) \times
\D_1 \times \ldots \times \D_n$, where $G$ is a compact connected Lie
group and the $\D_i$ are exotic $\Z_p$--root data. By the product
decomposition theorem \cite[Thm.~1.4]{dw:split} we have
an associated decomposition $BX \simeq BX' \times BX_1 \times \ldots
\times BX_n$, where $X'$ is a connected $p$--compact group with
$\D_{X'} \cong \D_G \otimes \Z_p$ and the $X_i$ are exotic $p$--compact
groups with $\D_{X_i} \cong \D_i$. 

We finish the proof by showing that $H^*(BX';\F_p) \cong
H^*(BG;\F_p)$ as $\F_p$--algebras. Lemma~\ref{conn} shows that
$\cC_{X'}(\nu)$ is connected for every elementary abelian $p$--subgroup
$\nu\co BE \to BX'$ of $X'$. If $\nu$ is {\em toral}, i.e., if $\nu$
factors through the maximal torus $BT \to BX'$ of $X'$, a result of
Dwyer--Wilkerson \cite[Thm.~7.6]{dw:center} (see also
\cite[Prop.~8.4(3)]{AG08classification}) shows that $\pi_0(\cC_{X'}(\nu))$ can be
computed from the root datum $\D_{X'}$. Since $\D_{X'} \cong \D_{G\pcom}$ it
follows that $\cC_{G\pcom}(\nu)$ is also connected for every toral
elementary abelian $p$--subgroup $\nu$ of $G\pcom$. If $E$ is an elementary
abelian subgroup of $G$, and $i\co E \to G$ is the inclusion we have
$C_G(E)\pcom \cong \cC_{G\pcom}(Bi)$, so $C_G(E)$ is connected for
every elementary abelian $p$--subgroup $E$ of $G$ which is contained in a
maximal torus. By results of Steinberg
\cite[Thms.~2.27~and~2.28]{steinberg75} and Borel
\cite[Thm.~B]{borel61} it follows that $G$ does not have $p$--torsion,
or equivalently that $H^*(BG;\F_p)$ is a polynomial algebra concentrated in
even degrees.
But, by \cite[Thm.~12.1]{AGMV08}, this implies that we have isomorphisms
$$
H^*(BX';\Z_p) \xrightarrow{\cong} H^*(BT;\Z_p)^W \xleftarrow{\cong}
H^*(BG;\Z_p)
$$
since the Weyl group $W$ and its action of $T$ is determined by
$\D_{X'} \cong \D_{G\pcom}$. In particular $H^*(BX';\F_p) \cong
H^*(BG;\F_p)$ as algebras over the Steenrod algebra, as wanted.
\end{proof}

The following proposition determining the compact Lie groups whose
classifying space has polynomial cohomology basically goes back to
Borel \cite[Thm.~2.5]{borel61}.

\begin{prop} \label{lieprop}
A finitely generated graded polynomial $\F_p$--algebra $A$ concentrated in even
degrees is isomorphic to $H^*(BG;\F_p)$ for some compact connected Lie
group $G$ if and only if the type of $A$ is union of multisets of
degrees occurring in Table~\ref{lietable}.
\end{prop}

\begin{proof}
In \cite[Thm.~2.5]{borel61} Borel determined for all simple simply
connected compact Lie groups $G$, the prime numbers~$p$ such that
$H^*(BG;\F_p)$ is a polynomial algebra concentrated in even degrees,
and this is exactly the data in
Table~\ref{lietable}. ($G=\Spin(2n+1)$, $p\geq 3$, is omitted since
$H^*(B\Spin(2n+1);\F_p) \cong H^*(B\Sp(n);\F_p)$ as $\F_p$--algebras.)

Assume now that $G$ is an arbitrary compact connected Lie group such
that $H^*(BG;\F_p)$ is a polynomial $\F_p$--algebra concentrated in
even degrees. There exists a fibration $BG' \to BG \to BS$, where $G'$
denotes the commutator
subgroup of $G$ and $S$ is a torus. Since $\pi_2(BG) \to \pi_2(BS)$ is
surjective by the long exact sequence in homotopy, $H^2(BS;\F_p) \to
H^2(BG;\F_p)$ is injective. Because $H^*(BS;\F_p)$ is generated by
elements in degree $2$ and $H^*(BG;\F_p)$ is polynomial, the
homomorphism $H^*(BS;\F_p) \to H^*(BG;\F_p)$ is injective and
$H^*(BG;\F_p)$ is free over $H^*(BS;\F_p)$. Hence the Eilenberg--Moore
spectral sequence for the fibration collapses and we get an
isomorphism $H^*(BG;\F_p) \cong H^*(BS;\F_p) \otimes H^*(BG';\F_p)$ of
$\F_p$--algebras. This shows that $H^*(BG';\F_p)$ is a polynomial
algebra concentrated in even degrees as well, so we can assume that
$G$ has finite fundamental group. But, since $H^3(BG;\F_p) = 0$,
$\pi_1(G)$ cannot contain $p$--torsion, so we can assume that $G$ is simply
connected. Hence $G$ splits as a product of simple simply
connected compact Lie groups, and in particular the type of
$H^*(BG;\F_p)$ is a union of multisets of degrees from Table~\ref{lietable}.
\end{proof}

The next proposition covers the exotic case.

\begin{prop} \label{exprop}
A finitely generated graded polynomial $\F_p$--algebra $A$ concentrated in even
degrees is isomorphic to $H^*(BY;\F_p)$ for $Y$ a product of exotic
$p$--compact groups if and only if the type of $A$ is union of
multisets of degrees occurring in Tables~\ref{extable1} and
\ref{extable2}.
\end{prop}

\begin{proof}
The degrees of every exotic $\Q_p$--reflection group, $p$
odd, is realized as the type of a polynomial $\F_p$--cohomology ring;
this is the culmination of the work of a number of people: Sullivan
\cite[p.~166--167]{sullivan05}, Clark--Ewing \cite{CE74}, Quillen
\cite[\S10]{quillen72}, Zabrodsky \cite[4.3]{zabrodsky84}, Aguad{\'e}
\cite{aguade89} and Notbohm--Oliver \cite{notbohm98}. See also
\cite{notbohm99} and \cite[\S7]{AGMV08} for a unified treatment. In
particular the multisets of degrees in Tables~\ref{extable1} and
\ref{extable2} are all realized.

Conversely suppose that $X$ is an exotic $p$--compact group,
with $H^*(BX;\F_p)$ a polynomial $\F_p$--algebra
concentrated in even degrees. As used before, this implies that
$H^*(BX;\Z_p)$ is also polynomial and $H^*(BX;\Z_p) \cong
H^*(BT;\Z_p)^W$. In particular, the type of $H^*(BX;\F_p)$ agrees with
the type of $H^*(BT;\Z_p)^W \otimes \Q$, i.e., with the degrees of
the $\Q_p$--reflection group $W$, which by assumption is exotic.
For $p=2$ the classification of finite $\Q_2$--reflection groups says that
there is only one exotic $\Q_2$--reflection group, namely $G_{24}$. By
the classification of
$\Z_p$--root data \cite[Thm.~8.1]{AG08classification}, there is a
unique $\Z_2$--root datum associated to $G_{24}$, namely $\D_X =
\D_{\DI(4)}$, the root datum of the exotic $2$--compact group $\DI(4)$
constructed by Dwyer--Wilkerson \cite{DW93}. Using the formula
for the component group of a centralizer \cite[Thm.~7.6]{dw:center} (see also
\cite[Prop.~8.4(3)]{AG08classification}) one sees that
$\cC_X({}_2\dT)$ is not connected, where ${}_2\dT$ denotes the maximal
elementary abelian $2$--subgroup of the discrete torus $\dT$. This
contradicts Lemma~\ref{conn}, so the classifying space of a
$2$--compact group with this root datum cannot have polynomial
$\F_2$--cohomology ring concentrated in even degrees, and we must have
$p$ odd. (Alternatively, one can see directly that the polynomial
algebra $\F_2[x_8, x_{12}, x_{28}]$ cannot be an unstable algebra over
the mod $2$ Steenrod algebra $\mathcal{A}_2$: For degree reasons the
ideal generated by $x_8$ and $x_{12}$ is closed under $\mathcal{A}_2$,
so the quotient $\F_2[x_{28}]$ would have to be an unstable algebra
over $\mathcal{A}_2$. However this is not possible since $28$ is not a
power of $2$.)

It only remains to show that for $p$ odd, the degrees of any exotic
$\Q_p$--reflection group which does not occur in Tables~\ref{extable1}
and \ref{extable2}, can be written as a multiset union of degrees of
those which do: The cases $G_4$, $G_5$, $G_6$, $G_7$, $G_{10}$, $G_{11}$,
$G_{13}$, $G_{15}$, $G_{18}$, $G_{19}$, $G_{25}$, $G_{26}$ and
$G_{27}$ are left out of Table~\ref{extable2} since these are covered
by the groups $G(6,3,2)$, $G(6,1,2)$, $C_4\times C_{12}$,
$C_{12}\times C_{12}$, $G(12,1,2)$, $C_{24}\times C_{24}$, $G_8$,
$G(12,1,2)$, $G(30,1,2)$, $C_{60}\times C_{60}$, $G(6,2,3)$,
$G(6,1,3)$ and $C_6 \times G_{20}$ respectively.
\end{proof}

\begin{proof}[{Proof of Theorem~\ref{mainthm}}]
Suppose that we are in the setup of the theorem, with a fixed
commutative Noetherian ring $R$ of finite Krull dimension. If $H^*(Y;R)$ is a
polynomial $R$--algebra, then Proposition~\ref{RtoFp} shows that for
$p\in\cP$, $H^*(Y;\F_p)$ is a polynomial $\F_p$--algebra with the same
type and $Y\pcom$ is the classifying space of a $p$--compact
group. Hence by Propositions~\ref{polyprop}, \ref{lieprop} and
\ref{exprop}, the type of $H^*(Y;R)$ satisfies the restrictions of
Theorem~\ref{mainthm}, as wanted.

Conversely, suppose that we are given a multiset of degrees, which satisfy
the degree restrictions for all $p \in \cP$. Then for each $p\in \cP$
there exists by Propositions~\ref{lieprop} and \ref{exprop} a space
$B_p$ such that $H^*(B_p;\F_p)$ is a polynomial $\F_p$--algebra with
the given type. By Proposition~\ref{glueing} we can construct a space
$Y$, such that $H^*(Y;R)$ is a polynomial $R$--algebra with the given
type.
\end{proof}

\begin{proof}[{Proof of Corollary~\ref{arithprog}}]
By inspection a given multiset of degrees occurs at most finitely many
times in Tables~\ref{lietable}--\ref{extable2}. It follows that a
given finite multiset $L$ can only be decomposed in finitely many ways
as a union of multisets of degrees from Tables~\ref{lietable}--\ref{extable2}
(ignoring for now restrictions on primes~$p$). We next have to
determine the prime numbers $p$ for which each of these decompositions
can occur.

For short, say that a set of integers is {\em listable} if
it can be written as a union of congruence classes modulo some number
$N$. Finite intersections and finite unions of
listable sets are again listable and there is an obvious
algorithm for computing them as listable sets.
Notice that each entry in the tables occur for a prime number
$p$ if and only if $p$ belongs to some listable set. (This follows directly
from the tables, together with the
observation that for any integer $k$ we have $p\geq k$ if and only if
$p$ belongs to a union of congruence classes modulo
$N=\prod_{l<k,\,l\,\,\text{prime}} l$; e.g., $p \geq 5$ iff $p \equiv
\pm 1$ mod $6$.) A given decomposition of $L$ (as a union of multisets
of degrees corresponding to entries in the tables),
can be used at a prime $p$ if and only if each of the
entries occur at $p$, and by the above this is equivalent to $p$
belonging to an explicit listable set depending only on the decomposition.
Since there are only finitely many decompositions of $L$ as above, 
we conclude that the set of primes $p$ for which there exists some
decomposition of $L$ at $p$ is a listable set. The result now follows
from Theorem~\ref{mainthm}.
\end{proof}

\begin{rem}
A different line of argument for Theorem~\ref{mainthm} in the case
$R=\F_2$ is as follows: As before, using product splitting for
$p$--compact groups, one can reduce to the case where $H^*(Y;\F_2)
\cong H^*(BX;\F_2)$ for $X$ a simply connected simple $2$--compact
group, and the type of $H^*(BX;\F_2)$ agrees with the degrees of
$W_X$. We hence have to rule out the degrees of $W(D_n)$ ($n\geq 4$),
$W(E_6)$, $W(E_7)$, $W(E_8)$, $W(F_4)$, $W(G_2)$, and $W(\DI(4))$. All
but the first family can be eliminated directly, even as algebras over
the Steenrod algebra, by \cite[Thm.~1.4]{thomas65}, and the same can
be seen to hold for $W(D_n)$, $n \geq 5$, $n$ odd, using
\cite{ST69}. The case $W(D_n)$, $n \geq 4$, $n$ even, can actually
exist as an algebra over the Steenrod algebra (when $n$ is a power of
$2$) but cannot occur as the cohomology ring of a space
by the following argument: Consider $X' =
\cC_X({}_2\dT)$. By \cite[Thm.~7.6]{dw:center} $W_{X'}$ is the subgroup of elements in
$W$ which act trivially on $L$ modulo $2$, and hence $W_{X'}$ is a
normal $2$--subgroup of $W$ (cf.~\cite[Lem.~11.3]{AGMV08}). By
assumption and Lemma~\ref{conn}, $X'$ is connected so $W_{X'}$ is
generated by reflections. Since all reflections in $W$ are
conjugate and $-1 \in W$ we conclude $W_{X'} = W$. In particular $W$
is a $2$--group, a contradiction.
\end{rem}

\begin{rem}
The spaces in Theorems~\ref{mainZ} and \ref{mainthm} can be realized
as cohomology rings of discrete groups, by a theorem of Kan--Thurston
\cite{KT76}, but the group can rarely be taken to be finite: If
$H^*(BG;R)$ is a finitely generated polynomial algebra, for $G$ a
finite group and $R$ any commutative ring, then all polynomial
generators are in degree $1$. This follows from a result of Benson--Carlson
\cite[Cor.~6.6]{BC94}, but one can also argue as follows: If all
generators are in even degrees, $H^*(BG;\F_p)$ is a polynomial ring with
the same type as $H^*(BG;R)$ for the primes $p$ which are not units
in $R$. Hence $BG\pcom$ is a connected $p$--compact group, and in
particular every element of $p$--power order in $G$ is divisible by $p$
(see \cite[Prop.~5.6]{DW94} and \cite[(4)]{mislin90}). Therefore the
order of $G$ is prime to $p$, and $H^*(BG;R) = R$. If there are
generators in odd degrees, then $R$ has to be an $\F_2$--algebra and a
small modification shows that all generators have to be in degree one
(see the proof of Proposition~\ref{polylemma}).
\end{rem}


\end{document}